\documentclass{amsart}
\usepackage[foot]{amsaddr}
\usepackage[english]{babel}
\usepackage{amsthm}

\usepackage{mathtools}
\usepackage{thmtools}
\usepackage{xfrac}
\usepackage{thm-restate}
\usepackage{enumerate}
\usepackage{hyperref}
\usepackage{xfrac}
\usepackage{cleveref}
\usepackage{verbatim}
\usepackage{amsmath}
\usepackage{amssymb}
\usepackage{mathabx}
\usepackage{graphicx}
\usepackage{amsfonts}
\usepackage{pb-diagram}
\usepackage{tikz-cd}
\usepackage{graphicx}
\usepackage{xcolor}
\usepackage{adjustbox}
\usepackage{mathtools}
\usepackage{todonotes}
\newtheorem{theorem}{Theorem}
\usepackage{mathrsfs}
\usepackage{bm}

\newtheorem{definition}[theorem]{Definition}

\newtheorem{remark}[theorem]{Remark}
\newtheorem{corollary}[theorem]{Corollary}
\newtheorem{lemma}[theorem]{Lemma}
\newtheorem{observation}[theorem]{Observation}
\newtheorem{fact}[theorem]{Fact}

\newtheorem{claim}[theorem]{Claim}

\newenvironment{subproof}[1][\proofname]{%
\begin{proof}[#1]%
	}{%
\end{proof}%
}

\newcommand{\p}[0]{\mathcal{P}}
\newcommand{\nin}[0]{\notin}

\newcommand{\Q}[0]{\mathcal{Q}}

\newcommand{\V}[0]{\mathcal{V}}
\newcommand{\W}[0]{\mathcal{W}}

\title[The homeomorphism group of a surface without boundary is minimal]{The compact-open topology on the homeomorphism group of a surface without boundary is minimal.}

\renewcommand{\emph}[1]{\textbf{#1}}

\author{J. de la Nuez Gonz\'alez }
\email{jnuezgonzalez@gmail.com}
\address{Korea Institute for Advanced Study (KIAS)}
\date{\today}
\thanks{Work supported by Samsung Science and Technology Foundation under Project Number SSTF-BA1301-51.  }
\begin{document}
\begin{abstract}
We show that the homeomorphism group of a surface without boundary does not admit a Hausdorff group topology strictly coarser than the compact-open topology. In combination with known automatic continuity results, this implies that the compact-open topology is the unique Hausdorff separable group topology on the group if the surface is closed or the complement in a closed surface of either a finite set or the union of a finite set and a Cantor set.
\end{abstract}

\maketitle
\renewcommand{\H}[0]{\mathcal{H}}
\newcommand{\T}[0]{\mathfrak{t}}
\newcommand{\tc}[0]{\T_{co}}
\newcommand{\nd}[0]{\mathcal{N}_{\T}(1)}
\newcommand{\cs}[0]{\mathcal{CS}}
\numberwithin{fact}{section}
\numberwithin{definition}{section}

\numberwithin{theorem}{section}
\numberwithin{lemma}{section}
\numberwithin{corollary}{section}
\numberwithin{observation}{section}
\numberwithin{claim}{section}

\setcounter{section}{-1}
\section{Introduction}
  
  We begin by recalling the following definition:
  \begin{definition}
  Let $G$ be a group and $\T$ a Hausdorff group topology on $G$. We say that $\T$ is minimal if $G$ admits no Hausdorff group topology strictly coarser than $\T$.
  \end{definition}
  The notion of minimality captures a key feature of compact groups and as such has received a substantial degree of attention in the literature. For the broader context and further questions we refer the reader to the comprehensive survey \cite{dikranjan2014minimality}.
  
  Given a topological space $X$, the compact-open topology on the homeomorphism group of $\H(X)$ of $X$ is the topology generated by all subsets of the form $[K,U]=\{f\in G\,|\,f(K)\subseteq U\}$, for $K\subseteq X$ compact and $U\subseteq X$ open in $X$. The following classical result is due to Arens \cite{arens1946topologies}. 
  \begin{fact}
  Let $X$ be a Hausdorff, locally compact and locally connected topological space. Then the compact-open topology is a group topology on $\H(X)$.
  \end{fact}
  
  Given a manifold $X$, we denote by $\H_{0}(X)$ the subgroup of $\H(X)$ consisting of all those $h\in\H(X)$ for which there exists an isotopy between $h$ and the identity: i.e. some continuous map $H:X\times [0,1]\to X$ such that $H(-,t)\in\H(X)$ for all $t\in[0,1]$, $H(x,0)=h(x)$ and $h(x,1)=x$. We denote by $\H_{c0}(X)\leq\H_{0}(X)$ the subgroup of all of those for which there exists some compact set $K\subseteq X$ such that $H(-,t)$ is supported on $K$ for all $t\in[0,1]$.
  The term surface will be used to refer to topological $2$-manifolds. The purpose of this short note is to provide a very elementary proof of the following result:
  \begin{theorem}
  \label{t: main}Let $X$ be a surface without boundary and $G$ be a group satisfying $\H_{c0}(X)\leq G\leq \H(X)$. Then the restriction to $G$ of the compact-open topology on $\H(X)$ is minimal.
  \end{theorem}
  
  It was first shown by Rosendal in \cite{rosendal2008automatic} that any group homomorphism from the homeomorphism group of a compact surface to a separable topological group is automatically continuous. This was later generalized by Mann to homeomorphism groups of compact manifolds in arbitrary dimension \cite{mann2016automatic} and later in \cite{mann2020automatic} to certain groups of homeomorphisms of non-compact manifolds. Combining these results with Theorem \ref{t: main} one obtains:
  
  \begin{corollary}
  Let $X$ be a closed surface and $F\subseteq X$ either a finite set (possibly empty) or the union of a finite set and a Cantor set. Then the compact-open topology is the unique separable Hausdorff group topology on the group $\H(X\setminus F)$.
  \end{corollary}
  Note that the compact-open topology was already known to be the unique complete separable group topology on the homeomorphism group of a compact manifold (see \cite{kallman1986uniqueness} and \cite{mann2016automatic}).
  The automatic continuity result in \cite{mann2020automatic} applies to the group of homeomorphisms of a manifold $X$ preserving a set $F$ as above. However, as discussed there, in dimension $2$ said group endowed with the compact-open topology is isomorphic to $\H(X\setminus F)$ as a topological group via the restriction map \footnote{That the two topologies are the same can be seen by applying the criterion of equality between $\tc$ and $(\tc)_{\restriction F}$ in Theorem $5$ from \cite{chang2017minimum} to the set-wise stabilizer of $F$ in $\H_{0}(X)$. }.
  
  In \cite{chang2017minimum} Chang and Gartside provide a counterexample to the minimality of the compact-open topology for any compact manifold $X$ with non-empty boundary which admits the following alternative description. The restriction homomorphism $\rho:\H(X)\to \H(int(X))$ is injective and continuous, where $int(X)=X\setminus\partial X$. Let $\T_{co}^{\partial}$ denote the preimage by $\rho$ of the compact-open topology on $\H(int(X))$ ($(\tc)_{\restriction \partial X}$ in the notation of \cite{chang2017minimum}). It is shown in \cite{chang2017minimum} that $\T^{\partial}_{co}\subsetneq \T_{co}$ if $X$ is compact. In fact, the same argument shows that this is also true if $X$ is not compact for $\T_{co}^{\partial}$ defined as above. On the other hand, since $\H_{c0}(int(X))\leq im(\rho)$, Theorem \ref{t: main} applies and thus we get:
  \begin{corollary}
  For any surface with boundary the topology $\T_{co}^{\partial}$ on $\H(X)$ is minimal.
  \end{corollary}

      \paragraph{\bf{Outline of the proof}}
        The proof of \ref{t: main} is structured as follows. We fix once and for all some surface $X$ without boundary, as well as some $G$ with $H\leq G\leq \H(X)$, where $H=\H_{c0}(X)$. We denote by $\tc$ the compact-open topology on $G$ and fix some group topology $\T$ on $G$ strictly coarser than $\tc$. Denote by $\nd$ the collection of open neighbourhoods of the identity in $\T$. Fix some arbitrary $g_{0}\in H\setminus\{1\}$ supported on a disk and assume the existence of $\V\in\nd$ with $g_{0}\nin\V$.
        
        Section \ref{s: preliminaries} establishes some notation and recalls some basic facts many readers will be familiar with. In Section \ref{s: neighbourhoods are rich} we see, using an auxiliary result from Section \ref{s: extension lemma}, that the assumption $\T\subsetneq\tc$ implies that every $\V\in\mathcal{N}_{\T}(1)$ is very rich, containing a plethora of fix-point stabilizers of embedded graphs.  Finally, in Sections \ref{untangling graphs},\ref{conclusion} we show this to be in contradiction with the existence of $\V$ as in the previous paragraph.

\section{Preliminaries}
  \label{s: preliminaries}
  We do not assume $X$ to be compact. However, we will often use the well-known fact that $X$ admits an exhaustion by compact submanifolds, which follows from a standard argument using the smoothability of surfaces and Whitney embedding theorem (see for instance \cite{milnor2016morse}).
  
  If we fix a metric $d$ on $X$ compatible with its topology, $\tc$ can be described as the topology of uniform convergence on compact sets. That is, a base of neighbourhoods of the identity for the compact-open topology on $G$ is given by the collection of sets
  $$
  \V_{K,\epsilon}=\{g\in G\,|\,\forall x\in K\,\,d(x,gx),d(x,g^{-1}x)<\epsilon\}
  $$
  where $K$ ranges over all compact subsets of $X$ and $\epsilon$ over all positive reals. Sometimes we might only be interested in the set $\V_{\epsilon}:=\V_{X,\epsilon}$ (potentially not open in $\tc$). For any subset $A\subseteq X$ and $\epsilon>0$ we let $N_{\epsilon}(A):=\{p\in X\,|\,d(p,A)<\epsilon\}$.
  We may and will assume that for any $p$ there is $\delta_{p}$ such that the ball $B(p,\delta_{p})$ is connected.
  
  Given a surface $Y$ and a closed set $F\subseteq Y$ we write $\H(Y,F)$ for the subgroup of all homeomorphisms
  of $Y$ that fixes $F$ point-wise, $\H_{0}(Y,F)$ for the subgroup of all homeomorphisms fixing $F$ isotopic to the identity by an isotopy that fixes $F$ at any point in time and $\H_{c0}(Y,F)$ for the group of all homeomorphisms isotopic to the identity through a compactly supported isotopy fixing $F$ at any point in time.\\
  
  
  
      \paragraph{\bf{Disks and arcs}}
        By a disk in a surface $Y$ we intend a homeomorphic image in $Y$ of a standard $2$-ball. We write $I$ for the interval $[-1,1]$. By an arc (in $Y$) from a point $p\in Y$ to a point $q\in Y$ we intend an injective continuous map $\alpha:I\to Y$ with $\alpha(-1)=p$ and $\alpha(1)=q$. We refer to $p$ and $q$ as the endpoints of $\alpha$. By a small abuse of notation, we may use the term $\alpha$ to refer to the image of $\alpha$. We will refer to $\alpha((-1,1))$ as the interior of $\alpha$.
        Whenever we concatenate a sequence of arcs or we restrict some arc to some interval $J\subseteq I$ we always assume an order-preserving reparametrization is applied at the end so that $I$ is again the domain of the resulting map.
        
        It is a consequence of the Jordan curve theorem and Sh\"onflies theorem that for any arc $\alpha\subseteq int(Y)$ there exists some homeomorphic embedding $\tilde{\alpha}:I\times I\to Y$ so that $\alpha$ is the restriction of $f$ to $\{0\}\times [-\frac{1}{2},\frac{1}{2}]\cong I$ \footnote{Working in the universal cover one can first extend the arc to a closed curve (using \cite{newman1939elements} p.164 or some other of the suggestions found at \href{https://mathoverflow.net/questions/57766/why-are-there-no-wild-arcs-in-the-plane}{https://mathoverflow.net/questions/57766/why-are-there-no-wild-arcs-in-the-plane}) and then apply Sch\"onflies theorem (see \cite{moise2013geometric}).}. We will refer to $\tilde{\alpha}$ as a rectangular extension of $\alpha$.
        By a regular path we mean the concatenation of finitely many arcs intersecting only at the endpoints.
        
        We say that two arcs $\alpha,\beta$, are transverse on an open set $U\subseteq Y$ if $\mathcal{I}:=U\cap\alpha\cap \beta$ is finite, coincides with $U\cap\alpha(\mathring{I})\cap\beta(\mathring{I})$ and $\alpha$ and $\beta$ cross at $q$ for any $q\in\mathcal{I}$ \footnote{That is, there is a disk with $p\in\mathring{D}$ and a homeomorphism $h:D\to B=\{(x,y)\in\mathbb{R}^{2}\,|\,|x|+|y|\leq 1\}$ taking $\alpha\cap D,\beta\cap D$ to $\{(x,y)\,|\,x=0\}\cap B$ and $\{(x,y)\,|\,y=0\}\cap B$ respectively.  }. We say that two collection of arcs $\mathcal{A},\mathcal{B}$ are transverse on $U$ if $\mathcal{A}$ is transverse on $U$ to each arc in $\mathcal{B}$. We will not mention $U$ explicitly when $U=X$. Sometimes we will talk about a collection of arcs being transverse to a certain $1$-dimensional compact submanifold $Z$, which simply means it is transverse to some decomposition of $Z$ into arcs.
        
        For the following see Chapters $2$ and $3$ in \cite{hirsch2012differential}.
        \begin{fact}
        \label{f: transversality}Let $Y$ be a surface and $\{\alpha_{i}\}_{i=}^{k},\{\beta_{j}\}_{j=1}^{r}$ two collections of arcs in $Y$ that can only meet $\partial Y$ at an endpoint and such that $\alpha_{i}\cap\alpha_{i'}$, $\beta_{j}\cap\beta_{j'}$ are finite for $1\leq i<i'\leq k$ and $1\leq j<j'\leq r$. Then for any neighbourhood $\W$ of the identity in the restriction of the compact-open topology to $\H(Y,\partial Y)$ there exists $\phi\in\W$ such that
        $\phi\cdot\alpha_{i}$ is transverse to $\beta_{j}$ on $int(Y)$ for all $1\leq i\leq k,1\leq j\leq r$.
        \end{fact}

        Given a compact surface $Y$ by a triangulation $\mathcal{T}$ of $Y$ we mean a homeomorphism between the geometric realization of some finite simplicial complex, which we assume to contain no double edges $Y$. We can think of it as a collection of disks, the triangles of $\mathcal{T}$ and of arcs, the edges of $\mathcal{T}$. For the following see \cite{moise2013geometric}.
        \begin{fact}
        Every compact surface admits a triangulation. Given two compact surfaces $Y,Y'$ with $Y\subseteq int(Y')$ every triangulation of $Y$ extends to a triangulation of $Y'$.
        \end{fact}
        
        We will often repeatedly use the following weak version of Alexander Lemma (see \cite{farb2011primer}, Chapter 4):
        \begin{fact}
        \label{f: alexander} If $D$ is a disk, then $\H(D,\partial D)=\H_{0}(D,\partial D)$. Therefore any homeomorphism of $X$ supported on an embedded disk is in $H$.
        \end{fact}
        
      \paragraph{\bf{Dehn twists}}
        
        An annulus $A$ in $X$ is the image of some homeomorphic embedding $h:S^{1}\to I\to X$. A core curve of $A$ is the image of $\alpha:S^{1}\to I\times S^{1}$, $\alpha(s)=(0,s)$ by some such $h$ with $im(h)=A$ and a Dehn twist over $A$ the homeomorphism of $X$ resulting from pushing forward the homeomorphism $(s,t)\mapsto (s,t+s)$ of $S^{1}\to I$ onto $A$ via some such $h$ an then extending it by the identity outside of $A$.
        \begin{observation}
        \label{o: dehn twist} Let $A$ be an annulus and $\gamma,\delta$ two disjoint arcs in $A$ joining the two boundary components of $A$. Let $\tau$ a Dehn twist over $A$. Then $\gamma\cup\tau^{2}\cdot\delta$ is connected and contains some core curve of $A$.
        \end{observation}

      \paragraph{\bf{Embedded graphs}} By an embedded graph $\Gamma$ we mean a finite tuple of arcs in $X$ such that for distinct $\gamma,\gamma'$ in $\Gamma$ any intersection point of $\gamma$ and $\gamma'$ is an endpoint of both and no two distinct $\gamma,\gamma'$ can have two endpoints in common. For convenience our notation will often treat $\Gamma$ as a mere set.
        
        We let $V(\Gamma)$ be the set of endpoints of $\Gamma$ in $X$. Alternatively, we say that $\Gamma$ is an embedded $\mathcal{Q}$-graph if $V(\Gamma)=\mathcal{Q}$. We write $\bigcup\Gamma=\bigcup_{\gamma\in\Gamma}\gamma$. We will refer to any neighbourhood of $\bigcup\Gamma$ simply as a neighbourhood of $\Gamma$. The group $G$ acts on the collection of embedded graphs by post-composition.
        
        Given embedded graphs $\Gamma=(\gamma_{i})_{i=1}^{k}$ and $\Gamma'=(\gamma'_{i})_{i=1}^{k'}$ we write
        $\Gamma\simeq\Gamma'$ if $k=k'$ and $\gamma'_{j}$ is an order preserving reparametrization of $\gamma'_{j}$ for $1\leq j\leq k$.
        
        We denote by $H_{\cup\Gamma}$ the subgroup consisting of all the elements of $H$ fixing $\bigcup\Gamma$ and by
        $H_{[\Gamma]}$ the subgroup of $H$ consisting of all $h\in H$ such that $h\cdot\Gamma\simeq\Gamma$.
        
        Given a neighbourhood $U$ of $\Gamma$ and a homeomorphic embedding $h:U\to X$ preserving the arcs of $\Gamma$ and fixing their endpoints we say that $h$ is orientation preserving at $\Gamma$ if for any $p\in\gamma$ there exists some homeomorphic embedding $\tilde{\gamma}:I \times I\to X$ with $\gamma=\tilde{\gamma}_{\restriction \{0\}\times I}$ and some neighbourhood $V$ of $p$ such that
        $V,h(V)\subseteq im(\tilde{\gamma})$ and if we let $A_{0}=\tilde{\gamma}([-1,0)\times I)$, $A_{1}=\tilde{\gamma}((0,1]\times I)$ then if $C\subseteq A_{i}$ for a component $C$ of $V\setminus\gamma$ then also $h(C)\subseteq A_{i}$.
        
        We denote by $H^{+}_{\cup\Gamma}$ and $H^{+}_{[\Gamma]}$ the subgroups of $H^{+}_{\cup\Gamma}$ and $H^{+}_{[\Gamma]}$ respectively consisting of those elements that are orientation preserving at $\Gamma$.\\
        
        
      \paragraph{\bf{Extending partial homeomorphisms}} The following facts can be seen as a consequence of Sch\"onflies theorem, the Jordan curve theorem and the classification of compact surfaces.
        
        \begin{fact}
        \label{f: transitive on disks} For any two families of disjoint embedded disks $\{D_{i}\}_{i=1}^{k},\{D'_{i}\}_{i=1}^{k}$ in some connected surface $Y$ and any collection of homeomorphisms $h_{i}:\partial D_{i}\cong \partial D'_{i}$, which we assume to be orientation-preserving if $Y$ is orientable, there exists $h\in \H_{c0}(Y)$ taking $D_{i}$ to $D'_{i}$ and restricting to $h_{i}$ on $\partial D_{i}$ for $1\leq i\leq k$.
        \end{fact}
        
        \begin{fact}
        \label{f: arcs within disk}Let $D$ be a disk and $\{\alpha_{i}\}_{i=1}^{k},\{\alpha'_{i}\}_{i=1}^{k}$ two families of pair-wise disjoint arcs between points in $\partial D$  such that either $\alpha'_{i},\alpha_{i}\subseteq\mathring{D}$ for all $i$ or the following holds
        \begin{itemize}
        \item for all $1\leq i\leq k$ the intersection of
        $\alpha_{i}$ with $\partial D$ consists of one or two endpoints $\alpha_{i}$ and the same is true for $\alpha'_{i}$
        \item if  $\alpha_{i}(-1)\in\partial D$, then $\alpha'_{i}(-1)=\alpha_{i}(-1)$ and the same is true if we replace $-1$ with $1$ and/or exchange the roles of $\alpha_{i}$ and $\alpha'_{i}$
        \end{itemize}
        Then there exists some $h\in \H(D,\partial D)$ such that $h\circ\alpha_{i}=\alpha'_{i}$ for all $i$.
        In particular, for any arc $\alpha$ and any embedded disk $D$ with $\alpha\subseteq\mathring{D}$ we have that $D$ is the image of a rectangular extension of $\alpha$.
        \end{fact}
        In particular, for any arc $\alpha$ and any embedded disk $D$ with $\alpha\subseteq\mathring{D}$ we have that $D$ is the image of a rectangular extension of $\alpha$.
        
        \begin{fact}
        \label{f: extensions around branching points}Let $D$ be a disk and $\Gamma$ some embedded $\mathcal{Q}$-graph such that $\bigcup\Gamma$ is simply connected and $\bigcup\Gamma\cap\partial D$ coincides with the set of points in $\mathcal{Q}$ belonging to a unique $\gamma\in\Gamma$. Then any homeomorphic embedding $h$ of $\bigcup\Gamma\cup\partial D$ into $D$ that is the identity on $\partial D$ extends to a homeomorphism of $D$.
        \end{fact}
        
      \paragraph{\bf{Point pushing maps}}
        Given an arc $\alpha$ from a point $p\in X$ to a point $q\in X$ and $\epsilon>0$ we denote by $\p_{\epsilon}(\alpha)$ the collection of all homeomorphisms that take $p$ to $q$ and are supported in some embedded disk $D$ with $\alpha\subseteq\mathring{D}\subseteq N_{\epsilon}(\alpha)$.
        The existence of rectangular extensions and Fact \ref{f: transitive on disks} implies that $\p_{\epsilon}(\alpha)\neq\emptyset$ for any $\epsilon>0$ and by Fact \ref{f: alexander} $\p_{\epsilon}(\alpha)\subseteq H$.
        If $\alpha$ is a regular path and $U$ an open set containing $im(\alpha)$ we let
        $\p_{\epsilon}(\alpha)$ be the collection of all products of the form $f_{k}f_{k-1}\dots f_{1}$, where
        $f_{i}\in\p_{\epsilon}(\alpha_{i})$ and
        $\alpha=\alpha_{1}*\alpha_{2}\dots \alpha_{k}$ for some decomposition of $\alpha$ into arcs. \\
        
      \paragraph{\bf{Bigons}}
        
        The following follows from a similar argument to that in the proof of 1.7 in \cite{farb2011primer}.
        \begin{fact}
        \label{f: bigon removals} Let $Y$ be a compact surface with boundary and let $\{\gamma_{i}\}_{i=1}^{k}$, $\{\gamma'_{i}\}_{i=1}^{r}$ be two families of pair-wise disjoint arcs in $Y$ between boundary points with
        $\gamma_{i}(\mathring{I}),\gamma'_{j}(\mathring{I})\subseteq int(Y)$. Assume that for each $1\leq i\leq k$ and $1\leq j\leq r$ the arcs $\gamma_{i}$ and $\gamma'_{j}$ are transverse on $int(Y)$ and some representative of the homotopy class of $\gamma'_{j}$ is disjoint from $\gamma_{i}$ on $int(Y)$.
        Then there is an innermost bigon bounded by $\Gamma$ and $\Gamma'$: an embedded disk $D$ in $Y$ whose boundary is the union of a subarc of some $\gamma_{i}$ and a subarc of some $\gamma'_{j}$ and whose interior is disjoint from all $\gamma_{i}$ and $\gamma'_{j}$.
        \end{fact}

\section{Extending partial homeomorphisms around graphs}
  \label{s: extension lemma}
  \begin{definition}
  Let $\Gamma$ be an embedded $\mathcal{Q}$-graph. By a nice system of disks around $\Gamma$ we mean a tuple $\mathcal{N}=(\{D_{q}\}_{q\in\Q},\{D_{\gamma}\}_{\gamma\in\Gamma},\{\theta_{\gamma}\}_{\gamma\in\Gamma})$ where $D_{q}$ and $D_{\gamma}$ are embedded disks and $\theta_{\gamma}:I\times I\cong D_{\gamma}$ such that
  \begin{itemize}
  \item $\{D_{q}\}_{q\in\mathcal{Q}}$ are pair-wise disjoint and $\{D_{\gamma}\}_{\gamma\in\Gamma}$ are pair-wise disjoint
  \item $D_{q}\cap D_{\gamma}$ is an arc if $q$ is an endpoint of $\gamma$ for $q\in\mathcal{Q}$ and $\gamma\in\Gamma$ and empty otherwise
  \item  $\theta_{\gamma}$ takes $\{0\}\times I$ to $\gamma\cap D_{\gamma}$ and $I\times\{-1\}$ and $I\times\{1\}$ to $D_{\gamma(-1)}\cap D_{\gamma}$ and $D_{\gamma(1)}\cap D_{\gamma}$ respectively.
  \end{itemize}
  and moreover for any arc arc $\gamma\in\Gamma$ from $q$ to $q'$ we have that
  \begin{itemize}
  \item $\gamma$ can be divided into three consecutive subarcs: $\gamma=\gamma_{q}*\gamma'*\gamma_{q'}$
  with $\gamma_{q}\subseteq D_{q}$, $\gamma'\subseteq D_{\gamma}$ and $\gamma_{q'}\subseteq D_{q'}$
  \item $\gamma$ intersects $\partial D_{q}$ only at the common endpoint of $\gamma_{q}$ and $\gamma'$, which lies in the interior of $D_{q}\cap D_{\gamma}$, and similarly for $q'$
  \end{itemize}
  We write $\bigcup\mathcal{N}=\bigcup_{q\in\mathcal{Q}}D_{q}\cup\bigcup_{\gamma\in\Gamma}D_{\gamma}$. It is easy to see that $\bigcup\mathcal{N}$ is a neighbourhood of $\bigcup\Gamma$. We will refer to any set of this form as a nice neighbourhood of $\Gamma$.
  \end{definition}
  
  \begin{definition}
  \label{d: star intersection}In particular, in the situation above each $D_{q}$ intersects each arc $\gamma\in\Gamma$ in either the empty set or a single arc from $q$ to $\partial D_{q}$ touching $\partial D_{q}$ in a single point. We will henceforth refer to any disk satisfying this condition for some $q\in V(\Gamma)$ as intersecting $\Gamma$ in a star\footnote{Note that this allows for the possibility that the disk in question intersects $\Gamma$ in just one or two arcs.} .
  \end{definition}
  
  \begin{remark}
  Given two collection of disks as in the definition, a collection of homeomorphisms $\{\theta_{\gamma}\}_{\gamma\in\Gamma}$ always exists.
  \end{remark}

  The following follows from the Jordan curve theorem and Sch\"onflies theorem by standard arguments:
  \begin{fact}
  \label{f: nice neighbourhoods}Let $\Gamma$ be an embedded $\Q$-graph. Then for any $\epsilon>0$ there is a nice system of disks
  $\mathcal{N}=(\{D_{q}\}_{q\in\Q},\{D_{\gamma}\}_{\gamma\in\Gamma},\{\theta_{\gamma}\}_{\gamma\in\Gamma})$ such that
  $\bigcup\mathcal{N}\subseteq N_{\epsilon}(\bigcup\Gamma)$ and $diam(D_{q})<\epsilon$ for all $q\in\Q$.
  \end{fact}
  
  The Lemma below is probably known, but we were unable to find a suitable reference.
  \newcommand{\h}[0]{\frac{1}{2}}
  \begin{lemma}
  \label{l: close to graph}Given any embedded graph $\Gamma$ and $\epsilon>0$ there exists some $\delta=\delta(\epsilon,\Gamma)>0$ such that for any neighbourhood $U$ of $\Gamma$ there is some closed neighbourhood $N$ of $\Gamma$ with $N\subseteq U$ and with the property that any homeomorphic embedding $h:N\to X$ which is the identity on $\bigcup\Gamma$, orientation preserving at $\Gamma$ and satisfies $d(p,h(p))<\delta$ for all $p\in N$ extends to some $\tilde{h}\in\V_{\epsilon}\cap H$. Moreover, if there exists some union $C$ of connected components of $X\setminus\bigcup\Gamma$ such that $h$ restricts to the identity on $C\cap N$, then we may assume that $\tilde{h}$ is the identity on $C$.
  \end{lemma}
  \begin{proof}
  Pick some nice system of disks  $\mathcal{N}=(\{D_{q}\}_{q\in\Q},\{D_{\gamma}\}_{\gamma\in\Gamma},\{\theta_{\gamma}\}_{\gamma\in\Gamma})$ around $\Gamma$ with $diam(D_{q})<\frac{\epsilon}{6}$ for all $q\in\Q$ and some $\eta>0$ such that $diam(\theta_{\gamma}(J_{1}\times J_{2}))<\frac{\epsilon}{6}$ for any $\gamma\in\Gamma$, $J_{1},J_{2}\subseteq I$, $diam(J_{1}),diam(J_{2})\leq 2\eta$.

  Pick $t_{0}=-1<t_{1}<\dots t_{m}<t_{m+1}=1$ such that
  $|t_{i}-t_{i+1}|\in (\eta,2\eta]$ for $0\leq i\leq m$.
  For any $1\leq i\leq m$ let $J_{i}=[\frac{t_{i}+t_{i-1}}{2},\frac{t_{i}+t_{i+1}}{2}]$.
  Let also $Q^{+}_{\gamma,i}=\theta_{\gamma}([0,\eta]\times J_{i})$ and $Q^{-}_{\gamma,i}=\theta_{\gamma}([-\eta,0]\times J_{i})$.

  Using compactness we can find some  $\delta\in(0,\frac{\eta}{2})$ such that
  \begin{enumerate}[(1)]
  \item \label{cond6} for any $\gamma\in\Gamma$ any arc in $X\setminus\bigcup\Gamma$ from a point
  in $\theta_{\gamma}([-1,0)\times I)$ to a point $\theta_{\gamma}((0,1]\times I)$ has diameter at least $3\delta$
  \item \label{cond4}$d(im(\theta_{\gamma}),im(\theta_{\gamma'}))\geq \delta$ for distinct $\gamma,\gamma'\in\Gamma$
  \item \label{cond5}$d(im(\theta_{\gamma}),D_{q})\geq \delta$ for $\gamma\in\Gamma$ and $q\in\Q$ such that $q$ is not an endpoint of $\gamma$
  \item \label{cond2}for any $\gamma\in\Gamma$ and
  $(s,s'),(t,t')\in I^{2}$ with $|s-t|+|s'-t'|\geq\frac{\eta}{2}$ we have $d(\theta_{\gamma}(t,t'),\theta_{\gamma}(s,s'))\geq 2\delta$
  \item \label{cond3}$d(D_{\gamma(-1)}\cup D_{\gamma_{(1)}},\theta_{\gamma}(I\times[-1+\frac{\eta}{2},1-\frac{\eta}{2}]))\geq 2\delta$
  \end{enumerate}
  
  It is easy to find some nice neighbourhood $N$ of $\Gamma$ contained in $U$ such that:
  \begin{enumerate}[(a)]
  \item\label{prop1} $d(N,\partial(\bigcup\mathcal{N}))\geq\delta$
  \item \label{prop2} for all $\gamma\in\Gamma$ we have $\theta_{\gamma}^{-1}(N\cap D_{\gamma})=[\xi,\xi]\times I$ for some $\xi>0$ such that $diam(\theta_{\gamma}[-\xi,\xi]\times\{t\})<\delta$ for all $t\in I$
  \end{enumerate}
  
  Assume now that $h:N\to X$ is a homeomorphic embedding restricting to the identity on $\Gamma$, orientation preserving at $\Gamma$ and such that $d(p,h(p))<\delta$ for all $p\in N$.
  Condition (\ref{prop1}) above implies that $h(N)\subseteq\bigcup\mathcal{N}$.
  
  It suffices to show that there exists some $g\in\V_{\frac{\epsilon}{2}}\cap H^{+}_{\cup\Gamma}$ such that $g\circ h$ extends to some $g'\in\V_{\frac{\epsilon}{2}}\cap H^{+}_{\cup\Gamma}$, since then $g^{-1}g'\in\V_{\epsilon}\cap H$ will be the extension we need.
  
  To begin with, notice that we may assume without loss of generality that $h\cdot\partial N$ is transverse to the boundary of the $Q_{\gamma,i}^{+}$ and $Q_{\gamma,i}^{-}$. For any $\gamma\in\Gamma$ and $1\leq i\leq m$
  condition (\ref{cond2}) above implies that
  $$h(\theta_{\gamma}([0,1]\times\{t_{i}\})),h(\theta_{\gamma}([-1,0]\times\{t_{i}\}))\subseteq B(\theta_{\gamma}(0,t_{i}),2\delta),$$ which together with the fact that $h$ is orientation preserving at $\Gamma$ implies that $h(\theta_{\gamma}([0,1]\times\{t_{i}\}))\subseteq Q_{\gamma,i}^{+}$, with only one endpoint on $\partial Q_{\gamma,i}^{+}$, and similarly for $\theta_{\gamma}([-1,0]\times\{t_{i}\})$ and $Q^{-}_{\gamma,i}$.
  We refer to the image of $[0,\eta]\times \{\frac{t_{i}+t_{i-1}}{2}\}$ and $[0,\eta]\times \{\frac{t_{i}+t_{i+1}}{2}\}$ by $\theta_{\gamma}$ as the two vertical sides of $Q^{+}_{\gamma,i}$, and similarly for $Q^{-}_{\gamma,i}$.
  
  \begin{claim}
  There is only one subarc of $h\cdot\partial N$ joining the two vertical sides of $Q^{+}_{\gamma,i}$ (resp. $Q^{-}_{\gamma,i}$), which passes through $h(\theta_{\gamma}((\xi,t_{i}))$ (resp. $h(\theta_{\gamma}(-\xi,t_{i}))$).
  \end{claim}
  \begin{subproof}
  Conditions (\ref{cond6}) and (\ref{prop2}) and the fact that $h$ is orientation preserving at $\Gamma$ imply that no point
  in $\theta_{\gamma}([-1,0)\times I)$ can be sent to $Q^{+}_{\gamma,i}$ by $h$. Together with conditions (\ref{cond4}) and (\ref{cond5}) this implies that no point outside of
  $D_{\gamma(-1)}\cup\theta_{\gamma}([0,1]\times I)\cup D_{\gamma(1)}$ can be sent to $Q^{+}_{\gamma,i}$ by $h$
  and a similar statement holds for $Q^{-}_{\gamma,i}$.
  
  On the other hand, there cannot exist $p\in  D_{\gamma(-1)}\cup\theta_{\gamma}(\{\xi\}\times [-1,t_{i}])=:A$ such that $h(p)\in\theta_{\gamma}([0,1]\times \{\frac{t_{i}+t_{i+1}}{2}\})=:B$, since by conditions (\ref{cond2}) and (\ref{cond3}) we have $d(A,B)\geq\delta$,  and the same holds if we exchange the role of the two endpoints of $\gamma$ and/or of $+$ and $-$. The moreover part is clear and the result follows.
  \end{subproof}
  
  This implies the existence of arcs $\alpha_{\gamma,i}^{+}\subseteq Q_{\gamma,i}^{+}$ and $\alpha_{\gamma,i}^{-}\subseteq Q_{\gamma,t_{i}}^{-}$ from $\theta_{\gamma}(\eta,t_{i})$ to $\theta_{\gamma}(\xi,t_{i})$ and from
  $\theta_{\gamma}(-\eta,t_{i})$ to $\theta_{\gamma}(-\xi,t_{i})$ respectively of which only the endpoints
  $\theta_{\gamma}(\xi,t_{i})$ and $\theta_{\gamma}(-\xi,t_{i})$ respectively belong to $h(N)$. Let $\beta^{+}_{\gamma,i}$ be the concatenation of $\alpha^{+}_{\gamma,i}$ and the arc with image $h(\theta_{\gamma}([0,\xi]\times\{i\}))$ and define $\beta^{-}_{\gamma,t}$ in a similar way. Fact \ref{f: arcs within disk} implies the existence of homeomorphisms $g_{\gamma,i}^{+}\in H$ supported on $Q^{+}_{\gamma,i}$ and
  $g_{\gamma,i}^{-}\in H$ supported on $Q^{-}_{\gamma,i}$ such that $g_{\gamma,i}^{+}(\beta^{+}_{\gamma,i})=\theta_{i}([0,\eta]\times\{t_{i}\})$
  and $g_{\gamma,i}^{-}(\beta^{-}_{\gamma,i})=\theta_{i}([-\eta,0]\times\{t_{i}\})$. Moreover, we may assume that
  $g_{\gamma,i}^{+}\circ h$ is the identity on $\theta_{\gamma}([0,\eta]\times\{t_{i}\})$ and similarly for $g_{\gamma,i}^{-}\circ h$. Since the $Q^{\pm}_{\gamma,i}$ have disjoint interiors and diameter at most $\leq\frac{\epsilon}{6}$ it follows that
  the product $g:=\prod_{\substack{\gamma\in\Gamma\\ 1\leq i\leq m}}g^{+}_{\gamma,i}g^{-}_{\gamma,i}$ is in $\V_{\frac{\epsilon}{6}}\cap H$.
  
  For all $q\in\Q$ let $D'_{q}$ is the union of $D_{q}$ and all the sets of the form
  $\theta_{\gamma}(I\times[-1,t_{1}])$ for $\gamma\in\Gamma$ with $\gamma(-1)=q$ and
  $\theta_{\gamma}(I\times[t_{m},1])$ for $\gamma\in\Gamma$ with $\gamma(1)=q$.
  Notice that $diam(D'_{q})\leq\frac{\epsilon}{2}$, since $diam(D_{q})\leq\frac{\epsilon}{6}$
  and
  $$diam(\theta_{\gamma}([-\eta,\eta]\times[-1,t_{0}]))\leq\frac{\epsilon}{6},$$ by the choice of $\eta$ and similarly for $\theta_{\gamma}(I\times[t_{m},1])$.

  Since $g$ is supported on $\bigcup\mathcal{N}$, the image of $N$ by the rectified map $g\circ h$ is still contained in $\bigcup\mathcal{N}$. Additionally, $g\circ h(N\cap E)=g\circ h(N)\cap E$
  whenever $E$ is either:
  \begin{itemize}
  \item $D'_{q}$ for some $q\in\Q$
  \item $\theta_{\gamma}([0,\eta]\times[t_{i},t_{i+1}])$ or $\theta_{\gamma}([-\eta,0]\times[t_{i},t_{i+1}])$  for  $\gamma\in\Gamma$ and $1\leq i\leq m-1$
  \end{itemize}
  and $g\circ h$ restricts to the identity on the arc $\partial E\cap N$. It follows from Fact \ref{f: arcs within disk} that $g\circ h$ extends to some $g'\in H$ which is the product of elements supported on sets $E$ as above. Since $diam(E)\leq\frac{\epsilon}{2}$ in both cases, it follows that $g'\in\V_{\frac{\epsilon}{2}}$. This concludes the proof. The moreover part is clear.
  
  \end{proof}

\section{Neighbourhoods of the identity contain fix-point stabilizers of embedded graphs}
  
  \newcommand{\subg}[1]{\langle #1 \rangle}
  
  \label{s: neighbourhoods are rich}We begin with the following observation:
  \begin{observation}
  \label{o: generation compact-open topology} There do not exist disks $D,D'$ with $D\subseteq D'$
  and $\mathcal{V}\in\nd$ such that $g\cdot D\subseteq D'$ for all $g\in \mathcal{V}$.
  \end{observation}
  \begin{proof}
  Indeed, given such $D,D'$, any $p\in X$ and any $\epsilon>0$ small enough applying Fact \ref{f: transitive on disks} two times yields some $h\in H$ such that  $p\in h\cdot \mathring{D}\subseteq h\cdot D'\subseteq B(p,\epsilon)$. Then $g\cdot D_{0}\subseteq B(p,\epsilon)$ for any $g\in\V^{h^{-1}}$, where $D_{0}=h\cdot D$. It follows easily from this using the definition of compactness that the conjugates of $\V$ by the action of $H$ generate a system of neighbourhoods of $\tc$ at the identity.
  \end{proof}
  
  Lemma \ref{l: mixing disks} below can be seen as a consequence of the theory of pseudo-Anosov mapping classes on compact surfaces. For the sake of self-containment and with the potential for higher dimension generalizations in view we provide a more elementary proof.
  
  \begin{lemma}
  \label{l: cheese} Suppose that $\alpha_{1},\alpha_{2},\dots \alpha_{m}$ are pair-wise transverse regular paths in $X$
  and $U_{1},\dots U_{m}$ open sets such that $\alpha_{l}(-1)\in U_{l}$. Denote by $\mathcal{I}$ the collection of self-intersection points of the $\alpha_{l}$ and of intersection points between different $\alpha_{l}$. Then for any $\mu>0$ there exists some $f_{l}\in H$ such that if we let $f=f_{m}f_{m-1}\dots f_{1}$
  then $\alpha_{l}\setminus N_{\mu}(\mathcal{I})\subseteq f\cdot U_{l}$ for $1\leq l\leq m$.
  \end{lemma}
  \begin{proof}
  Consider first the case $m=1$. Let $\alpha=\alpha_{1}$, $U_{1}=U$ and $p=\alpha(-1)$. Let $\alpha=\alpha^{1}*\dots \alpha^{k}$ be a decomposition of $\alpha$ into arcs not containing any point of $\mathcal{I}$ in their interior and write $\bar{\alpha}^{i}=\alpha^{1}*\dots \alpha^{i}$.
  Let $\hat{\alpha}^{i}=\alpha^{i}\setminus N_{\frac{\mu}{2}}(\mathcal{I})$ and pick some $\nu>0$ be smaller than
  $\frac{\mu}{2}$ and $d(\hat{\alpha}^{i},\hat{\alpha}^{j})$ for  $1\leq i<j\leq m$.

  We choose $g_{i}\in\p_{\nu}(\alpha^{i})$ by induction so that if we write $\bar{g}_{i}=g_{i}g_{i-1}\dots g_{1}$ ($\bar{g}_{0}=1$) then for all $0\leq i\leq m$ the set $\bar{g}_{i}\cdot U$ contains some path $\beta^{i}$ from $\bar{B}_{\nu}(p)$ to $\bar{\alpha}^{i}(1)$ such that $\bar{\alpha}^{i}\setminus N_{2\nu}(\mathcal{I})\subseteq\beta^{i}$ and $\alpha^{i}\subseteq\beta^{i}$ if $i\geq 1$.
  
  In the base case we simply take as $\beta^{0}$ some non-trivial arc in $U$ ending in $\alpha(-1)$ such that
  $\beta^{0}*\alpha^{0}$ is still an arc. Suppose now that $i\geq 1$ and the result has been shown for $i-1$.
  Let $\tilde{\alpha}^{i}$ be a rectangular extension of $\alpha^{i}$ with $im(\tilde{\alpha}^{i})\subseteq N_{\nu}(\alpha^{i})$ and let $D=im(\tilde{\alpha}^{i})$. We may assume that $D\cap\beta^{i-1}$ is a single arc $\gamma$ from $\partial D$ to $\alpha^{i}(-1)$ (use a rectangular extension of an arc $\gamma'*\beta$, where $\gamma'$ is a subarc of $\beta^{i-1}$). By Fact \ref{f: arcs within disk} there exist some $g_{i}$ supported in $D$ such that $g_{i}\cdot\gamma=\gamma*\alpha^{i}$.
  If $1\leq i\leq m-1$, the construction ensures that $\alpha^{i}\setminus N_{2\nu}(\mathcal{I})\subseteq\bar{g}_{i+1}\cdot U$, from it which it follows that
  $$\bar{g}_{m}\cdot U\supseteq (g_{i+1}\cdot U)\setminus\bigcup\nolimits_{j=i+2}^{m}N_{\nu}(\alpha^{j})\supseteq \alpha^{i}\setminus N_{2\nu}(\mathcal{I})$$
  
  For the general case one can proceed similarly. Take $\nu=\min\{\frac{1}{2}\mu,d(\hat{\alpha}_{l},\hat{\alpha}_{l'})\}_{l'\neq l}$, $\hat{\alpha}_{l}=\alpha_{l}\setminus N_{\nu}(\mathcal{I})$. We may assume that $U_{l}\cap\bigcup_{l'\neq l}N_{\nu}(\alpha_{l})=\emptyset$. For $1\leq l\leq m$ let $\mathcal{I}_{l}$ be the set of self-intersection points of $\alpha_{l}$ and pick some $f_{l}\in\p_{\nu}(\alpha_{l})$ such that
  $\alpha_{l}\setminus N_{\mu}(\mathcal{I}_{l})\subseteq f_{l}\cdot U_{l}$. Let $\bar{f}_{l}=f_{l}f_{l-1}\cdots f_{1}$. Then $\bar{f}_{l}(U_{l})=f_{l}(U_{l})$ and
  \[
  f\cdot U_{l}\supseteq (\bar{f}_{l}\cdot U_{l})\setminus\bigcup\nolimits_{l'> l}N_{\nu}(\alpha_{l'})\supseteq \alpha_{l}\setminus N_{\mu}(\mathcal{I}).
  \]
  \end{proof}

  \begin{lemma}
  \label{l: mixing disks}Let $D,E_{1},\dots E_{k}$ be disks in $X$ such that $E_{i}$ is not contained in $D$ for any $1\leq i\leq k$. Let also $K\subseteq X$ be a compact subset and $\epsilon$ a positive real. Then there exists $h,h'\in H$ fixing $D$ such that for any $1\leq i,j\leq k$ and any connected component $C$ of the complement of $h\cdot E_{i}\cup h'\cdot E_{j}$ either:
  \begin{itemize}
  \item $K\cap C=\emptyset$
  \item $diam(C)<\epsilon$
  \item $C\subseteq N_{\epsilon}(D)$
  \end{itemize}
  \end{lemma}
  \begin{proof}
  For $1\leq i\leq k$ choose some $p_{i}\in \mathring{E_{i}}\setminus D$.
  Pick some compact submanifold $Y$ such that $N_{\epsilon}(D),N_{\epsilon}(K)\subseteq Y$ and let $\mathcal{T}$ a triangulation of $Y$. We can choose $\mathcal{T}$ so that each triangle has diameter at most $\frac{\epsilon}{2}$ and $\{p_{i}\}_{i=1}^{k}\cap\mathscr{V}_{D}=\emptyset$. Let $\mathcal{T}_{D}$ be the collection of triangles disjoint from $D$ and $\mathscr{V}_{D}$ be the collection of all their vertices.
  Given adjacent $u,v\in\mathscr{V}_{D}$ we denote by $[u,v]$ the corresponding triangle side, an arc in $X$.
  
  Let $\eta=\min\{\frac{\epsilon}{2},\frac{1}{3}d(v,v')\,|\,v,v'\in\mathscr{V}_{D},v\neq v'\}$.
  For each $v\in\mathscr{V}_{D}$ choose disks $F_{v},F'_{v}$ such that $v\in \mathring{F}_{v}$, $F_{v}\subseteq \mathring{F}'_{v}$ and $F'_{v}\subseteq\mathring{F}''_{v}\subseteq B(v,\eta)$ and let $A_{v}$ be the annulus $F'_{v}\setminus\mathring{F}_{v}$.
  
  For $1\leq i\leq k$ it is easy to find some regular path $\alpha_{i}$ in $Y\setminus D$ starting at $p_{i}$ and such that for each triangle $T$ in $\mathcal{T}$ that does not intersect $D$ and each side $[u,v]$ of $T$ there exists some arc $\alpha_{i}^{[u,v]}\subseteq\alpha_{i}\cap\mathring{T}$ from a point $u'\in \mathring{F}_{u}$ to a point $v'\in \mathring{F}_{v}$. We may also make the choice in such a way that the $\alpha_{i}$ are pair-wise transverse and that if we denote by $\mathcal{I}$ the set consisting of all self-intersection points of $\alpha_{i}$ for some $i$ and of intersection points of $\alpha_{i}$ and $\alpha_{j}$ for different $i,j$, then $\mathcal{I}$ is contained in $\bigcup_{v\in\mathscr{V}_{D}}\mathring{F}_{v}$ and disjoint from all the $\alpha_{i}^{[u,v]}$.
  
  Lemma \ref{l: cheese} applied to the collection  of paths $\{\alpha_{1}\}_{i=1}^{k}$ with a suitably small constant $\mu$  provides some $f\in H$ fixing $D$ such that for each $1\leq i\leq k$ and each edge $[u,v]$ in a triangle $T$ in $\mathcal{T}_{D}$ we have $\alpha_{i}^{[u,v]}\subset f\cdot E_{i}$.
  
  For each $v\in\mathscr{V}_{D}$ let $\tau_{v}\in H$ be a Dehn twist over the annulus $A_{v}$. Let $\tau=\prod_{v\in\mathscr{V}_{D}}\tau^{2}_{v}$ and $h'=\tau h$. Then by virtue of Observation \ref{o: dehn twist} for each $1\leq i, j\leq k$ the set
  $E_{i,j}:=h\cdot E_{i}\cup h'E_{j}$ contains:
  \begin{itemize}
  \item a core curve $\beta_{v}$ of each of the annuli $A_{v}$, $v\in\mathscr{V}_{D}$
  \item for each edge $[v,v']$ of some triangle $T$ in $\mathcal{T}_{D}$ some path from
  $\beta_{v}$ to $\beta_{v'}$ in $\mathring{T}$
  \end{itemize}
  It easily follows that every connected component in the complement of $E_{i,j}$ is either contained in
  $N_{\epsilon}(D)$ or it has diameter less than $\epsilon$ or else it is contained in $N_{\epsilon}(Y^{c})$ and is thus disjoint from $K$.
  \end{proof}
  
  \begin{corollary}
  \label{c: mixing disks}Let $D,E_{1},\dots E_{k}$ be disks in $X$.
  Assume that $E_{i}\nsubseteq D$ for $1\leq i\leq k$ and there exists $\V\in\nd$ such that for all $g\in\V$ there is $1\leq i\leq k$ with $g\cdot E_{i}\cap D=\emptyset$.
  Then for any $\epsilon>0$ and any compact set $K\subseteq X$ there is $\V_{D}^{K,\epsilon}\in\nd$ such that for any $g\in\V_{D}^{K,\epsilon}$ either:
  \begin{itemize}
  \item $g\cdot D\subseteq N_{\epsilon}(D)$
  \item $diam(g\cdot D)<\epsilon$
  \item $g\cdot D\cap K=\emptyset$
  \end{itemize}
  \end{corollary}
  \begin{proof}
  Simply let $\mathcal{V}_{D}^{K,\epsilon}=\mathcal{V}^{h^{-1}}\cap\mathcal{V}^{(h')^{-1}}$, where $h,h'$ are as given by Lemma \ref{l: mixing disks} applied to $D,E_{1},\dots E_{k}$.
  \end{proof}

  \begin{lemma}
  \label{l: spread all over}Let $D,E_{1},\dots E_{k}$ be disks in $X$ and $\V\in\nd$. Then for any $\V\in\nd$ there exists $g\in\V$ such that $g\cdot D\cap E_{i}\neq\emptyset$ for all $1\leq i\leq k$.
  \end{lemma}
  \begin{proof}
  Suppose that $\V\in\nd$ fails to satisfy the property. Up to making $D$ smaller, we may assume that $E_{i}\nsubseteq D$ for all $1\leq i\leq k$.
  
  Using \ref{f: transitive on disks} take $h\in G$ such that $supp(g_{0})\subseteq h\cdot D=:D'$, let $\V'=\V^{h^{-1}}$ and pick some compact set $L\subset X$ and $\epsilon>0$ such that $\V_{L,\epsilon}\subseteq\V'$ and $N_{\epsilon}(D')$ is contained in a disk $D''$. Let also $\V_{1}=\V_{1}^{-1}\in\nd$ be such that $g_{0}\nin\V_{1}^{3}$ .
  Consider the intersection $\V_{0}:=\V_{1}\cap\V_{D'}^{L,\epsilon}$, where $\V_{D'}^{L,\epsilon}$ is given by Corollary \ref{c: mixing disks} applied to $\V'$ and $D',h\cdot E_{i}$.
  For any given $g\in \V_{0}$ at least one of the following possibilities holds:
  \begin{itemize}
  \item $g\cdot D'\subseteq D''$
  \item $diam(g\cdot D')<\epsilon$
  \item $g\cdot D'\cap L=\emptyset$
  \end{itemize}
  If $g\in\V_{0}$ satisfies the second or third possibility, then $g_{0}^{g^{-1}}\in\V_{L,\epsilon}\subseteq
  \V_{1}$ and thus $g_{0}\cdot\V_{1}^{3}$, contradicting the choice of $\V_{1}$. Hence the first alternative must always hold, contrary to Observation \ref{o: generation compact-open topology}.
  %
  \end{proof}
  
  \begin{lemma}
  \label{l: nerve} For any $\V\in\nd$ and any compact set $L\subseteq X$ there exists some embedded graph $\Gamma$ such that $H^{+}_{\cup\Gamma}\subseteq\V$ and $L$ is contained in the closure of one connected component $U_{0}$ of $X\setminus\bigcup\Gamma$.
  \end{lemma}
  \begin{proof}
  Take $\V_{0}=\V_{0}^{-1}\in\nd$ with $\V_{0}^{9}\subseteq \V$. Let $K$ a compact set and $\epsilon>0$ be such that $\V_{K,\epsilon}\subseteq\V_{0}$.  Now, pick some disk $D$ with $diam(D)<\epsilon$ and some triangulation $\mathcal{T}$ of a compact submanifold $Y$ of $X$ containing $N_{\epsilon}(K\cup L)$ in which each triangle has diameter strictly less than $\epsilon$. Applying Lemma \ref{l: spread all over} we can find some
  $g'\in\V_{0}$ such that $g'\cdot D\cap\mathring{T}\neq\emptyset$ for each triangle $T$ in $\mathcal{T}$. Let $Y'$ be a compact submanifold such that $g\cdot D\subseteq int(Y')$, equipped with a triangulation $\mathcal{T}'$ that restricts to $\mathcal{T}$ on $Y$. We may also assume that $g'\cdot \partial D$ is transverse to all the internal edges of $\mathcal{T}'$, so that for any triangle $T\in\mathcal{T}'$ if it intersects $T$ if and only if it intersects $\mathring{T}$.     
  
  Let $F$ be the union of all triangles of $\mathcal{T}'$ intersecting $g\cdot D$ and $\mathscr{U}$ the collection of vertices of $\mathcal{T}'$ that lie in $\mathring{F}$.
  \begin{claim}
  \label{c: surrounding vertices}There is $\phi\in\V_{K,\epsilon}$ such that $\phi g\cdot D$ intersects exactly the same triangles of $\mathcal{T}'$ as $g\cdot D$ and $\mathscr{U}\subseteq \phi g\cdot\mathring{D}$.
  \end{claim}
  \begin{subproof}
    
  Indeed, we can easily find a family of disjoint disks $\{D_{u}\}_{u\in\mathscr{U}}$ contained in $\mathring{F}$ and $\{\phi_{u}\}_{u\in\mathscr{U}}\subseteq H$, where $\phi_{u}$ is supported on $D_{u}$, $u\in\phi_{u}g\cdot D$ and $D_{u}$ is either disjoint from $K$ or is contained in some small neighbourhood of some triangle $T\in\mathcal{T}$ and has diameter less than $\epsilon$. We then let $\phi=\prod_{u\in\mathscr{U}}\phi_{u}$.
  \end{subproof}

  Using Fact \ref{f: transversality} we can slightly perturb $g:=\phi g'$ within $\V_{0}^{2}$ so that $g\cdot\partial D$ is transverse to the edges of $\mathcal{T}'$, while preserving the conclusion of Claim \ref{c: surrounding vertices}.
  
  We construct an embedded graph $\Gamma$ in $X$ as follows. For each triangle $T$ in $\mathcal{T}'$ such that $T\cap g\cdot D\neq\emptyset$ and every connected component $C$ of
  $T\setminus g\cdot D$ we pick a vertex $q_{C}\in C\cap\mathring{C}$ and given two components $C,C'$ in adjacent triangles $T,T'$ satisfying $C\cap C'\neq\emptyset$ add an arc between $v_{C}$ and $v_{C'}$ inside $T\cup T'$ intersecting $C\cap C'$ in a single point. The choice can clearly be made in such a way that two of the resulting arcs can only intersect at a common endpoint. We can assume that $\mathcal{T}$ contains at least $3$-triangles so that no two of the resulting edges can have the same pair of endpoints.
  
  Let $U_{0}$ be the connected component of $X\setminus\bigcup\Gamma$ containing $g\cdot D$.
  
  \begin{claim}
  $K\cup L\subseteq\bar{U_{0}}$
  \end{claim}
  \begin{subproof}
  The inclusion $N_{\epsilon}(K\cup L)\subseteq Y$ implies that for any triangle $T$ in $\mathcal{T}$ with $(K\cup L)\cap T\neq\emptyset$ the triangle $T$ and any triangle in $\mathcal{T}$ adjacent to it must intersect $g\cdot D$ (be contained in $F$) and thus that all the vertices of $T$ must belong to $\mathscr{U}$, since they are all in $int(Y)$. It follows that for each component $C$ of $T\setminus\bigcup g\cdot D$ and every edge $\gamma\in\Gamma$ with $v_{C}$ as an endpoint and crossing some side $\sigma$ of $\partial T$ there is some subarc $\sigma'\subseteq\sigma$ intersecting $g\cdot D$ only in $\sigma'(-1)$ and intersecting $\bigcup\Gamma$ only in $\sigma(1)\in\gamma\cap\sigma$. It follows that any component $V$ of $T\setminus\bigcup\Gamma$ intersects $g\cdot D$ and therefore that $K\cup L\subseteq\bar{U_{0}}$.
  \end{subproof}
  
  \begin{claim}
  \label{c: out of neighbour}For any neighbourhood $N$ of $\bigcup\Gamma$ there exists some $f\in\V_{K,\epsilon}$ such that $f\cdot (U_{0}\setminus N)\subseteq g\cdot D$.
  \end{claim}
  \begin{subproof}
  It is easy to see that $U_{0}\setminus g\cdot D$ is the union of all the connected components of $T\setminus(\bigcup\Gamma\cup g\cdot D)$ bordering $g\cdot D$ as $T$ ranges among all the triangles of $\mathcal{T}'$ that intersect $g\cdot D$ non-trivially but are not contained in $g\cdot D$.
  We construct $f$ as a homeomorphism preserving each of the triangles in $\mathcal{T}'$ intersecting $g\cdot D$ and fixing their complement in $X$.	We can first define $f$ on the $1$-skeleton of $\mathcal{T}'$, fixing $u,v$ and preserving $[u,v]\cap\bigcup\Gamma$ for every edge $[u,v]$ and mapping $[u,v]\setminus N$ into $[u,v]\cap g\cdot D$ and then use \ref{f: arcs within disk} to find an extension to the interior of the triangles with the same property.
  \end{subproof}
  
  Take an arbitrary $h\in H^{+}_{\cup\Gamma}$ and let
  $h_{0}\in H^{+}_{\cup\Gamma}$ the map that agrees with $h$ on $U_{0}$ and is the identity outside of $\bar{V}_{0}$. Let $\delta=\delta(\Gamma,\epsilon)$ be the constant provided by Lemma \ref{l: close to graph}. Continuity of $h_{0}$ implies the existence of a neighbourhood $V$ of $\Gamma$ such that $d(p,h\cdot p)<\delta$ for all $p\in V$. Lemma \ref{l: close to graph} then provides some $h_{1}\in \V_{\epsilon}\cap H\subseteq
  \V_{0}$ agreeing with $h_{0}$ on $N\cup(X\setminus U_{0})$, where $N\subseteq V$ is some neighbourhood of $\Gamma$.
  
  Then Claim \ref{c: out of neighbour} provides some $f\in\V_{\epsilon}\cap H^{+}_{\cup\Gamma}$ such that
  $f\cdot (U_{0}\setminus N)\subseteq g\cdot D$ so that
  $$g^{-1}f\cdot supp(h_{1}^{-1}h_{0})\subseteq g^{-1}f\cdot (U_{0}\setminus N)\subseteq D.$$
  It follows that $h_{0}\in h_{1}\V_{\epsilon}^{g^{-1}f}\subseteq\V_{0}^{8}$. On the other hand, $h_{0}^{-1}h\in H$ is in $\V_{0}$, since its support is disjoint from $K$, and thus $h\in\V_{0}^{9}\subseteq\V$.
  \end{proof}
  
  \begin{corollary}
  \label{c: control on branch points}For any $\V\in\nd$ and any open $U\subseteq X$ there exists some $\mathcal{Q}$-embedded graph $\Gamma$ with $\mathcal{Q}\subseteq U$ such that $H^{+}_{\cup\Gamma}\subseteq\V$. Moreover, we can assume $\Gamma$ to be transverse to any given finite collection of arcs.
  \end{corollary}
  \begin{proof}
  We first observe that for any finite set $\mathcal{F}$ of points, any ball $B=B(p,\delta)$ and any
  $\W\in\nd$ there exists some $g\in \W$ such that $g\cdot\mathcal{F}\subseteq B$. Indeed, take
  $\W_{0}=\W_{0}^{-1}\in\nd$ with $\W_{0}^{2}\subseteq\W$ and by Lemma \ref{l: nerve} some embedded graph $\Delta$ such that $H^{+}_{\bigcup\Delta}\subseteq \W_{0}$ and $\mathcal{F}\cup B\subseteq\bar{V}$ for some connected component $V$ of $X\setminus\bigcup\Delta$. Since $\T\subset\tc$ there exists some $g_{0}\in\W_{0}$ such that
  $g_{0}\cdot\mathcal{F}\subseteq U$ and then some $g_{1}\in H^{+}_{\cup\Gamma}$ such that $g_{1}\cdot(g_{0}\cdot\mathcal{F})\subseteq B$ by Fact \ref{f: transitive on disks}.
  
  Let $\V_{0}=\V_{0}^{-1}\in\nd$ be such that $\V_{0}^{3}\subseteq\V$. The previous Lemma provides a $\mathcal{Q}'$-embedded graph $\Gamma'$ such that $H_{\cup\Gamma'}\subseteq\V_{0}$. Since $\V_{0}$ is open in $\tc$ we can take $g\in\V_{0}$ such that $g\cdot \mathcal{Q}\subseteq U$. Up to perturbing $g$ slightly within $\V_{0}$ we may assume $\Gamma:=g\cdot\Gamma'$ is transverse to any given finite set of arcs, while $H^{+}_{\cup\Gamma}=(H_{\cup\Gamma'}^{+})^{g^{-1}}\subseteq\V$.
  \end{proof}
  
  \begin{lemma}
  \label{l: many point pushing maps}Let $\alpha$ be an arc and $\epsilon>0$. Then for any $\V\in\nd$ we have $\p_{\epsilon}(\alpha)\cap\V\neq\emptyset$..
  \end{lemma}
  \begin{proof}
  First consider any disk $D$ whose boundary is subdivided into four sides $\beta_{i}, 1\leq i\leq 4$, where $\beta_{i}$  is adjacent to $\beta_{i+1}$ (cyclically).
  If $\Gamma$ is a $\mathcal{Q}$-embedded graph transverse to the $\beta_{i}$, then it is easy to see that at least one of the following alternative holds:
  \begin{itemize}
  \item $D\cap\mathcal{Q}\neq\emptyset$
  \item some connected component of $D\setminus\bigcup\Gamma$ intersects
  $\beta_{1}(\mathring{I})$ and $\beta_{3}(\mathring{I})$
  \item some connected component of $D\setminus\bigcup\Gamma$ intersects
  $\beta_{2}(\mathring{I})$ and $\beta_{4}(\mathring{I})$
  \end{itemize}
  \begin{claim}
  For any $\V\in\nd$ there exists some embedded graph $\Gamma$ with $H^{+}_{\cup\Gamma}\subseteq\V$ which is transverse to the $\beta_{i}$ and satisfies the second alternative above.
  \end{claim}
  \begin{subproof}
  Suppose not. Notice that by Fact \ref{f: transitive on disks} there exists some $f\in H$ preserving $D$ and mapping $im(\beta_{i})$ homeomorphically onto $im(\beta_{i+1})$ (cyclically). If we let $\V_{0}=\V\cap\V^{f}$, then
  for any $\Gamma$ transverse to the $\beta_{i}$ such that
  $H^{+}_{\cup\Gamma}\subseteq\V_{0}$ necessarily the first alternative must take place, but this contradicts Corollary \ref{c: control on branch points}.
  \end{subproof}
  Now, given $\alpha$ and $\epsilon>0$ as in the premise,
  choose some embedded disk $E$ with $\alpha\subseteq\mathring{E}\subseteq E\subseteq N_{\epsilon}(\alpha)$, $\V_{0}=\V_{0}^{-1}\in\nd$ with $\V_{0}^{3}\subseteq\V$ and $\delta>0$ such that
  \begin{itemize}
  \item $B(p,\delta),B(q,\delta)$ are connected
  \item $B(p,\delta)\cup B(q,\delta)\subseteq E$
  \item $\V_{2\delta}\subseteq\V_{0}$
  \item $d(p,q)\geq 3\delta$
  \end{itemize}
  Now choose an embedded disk $D\subseteq\mathring{E}$ bounded by $\{\beta_{i}\}_{i=1}^{4}$ as above so that, $p\in\beta_{1}(\mathring{I})$, $q\in\beta_{3}(\mathring{I})$ and $diam(\beta_{1}),diam(\beta_{3})<\delta$.
  
  If $\Gamma$ satisfies $H^{+}_{\cup\Gamma}\subseteq\V_{0}$ and is as in the second alternative, then for $\eta>0$ small enough and some arc $\kappa$ in $D$ from some $p'\in\beta_{1}(\mathring{I})$ to $q'\in\beta_{3}(\mathring{I})$ we have $N_{\eta}(im(\kappa))\subseteq E\setminus\bigcup\Gamma$.
  On the other hand, there is	$h$ supported in the disjoint union $B(p,\delta)\cap B(q,\delta)$ such that $h\cdot(p,q)=(p',q')$. Then $h\in\V_{2\delta}$ and
  $$\p_{\eta}(\kappa)^{h}\subseteq (H^{+}_{\cup\Gamma})^{h}\subseteq\V_{0}^{h}\subseteq\V,$$
  while on the other hand
  $\emptyset\neq\p_{\eta}(\kappa)^{h}\subseteq\p_{\epsilon}(\alpha)$.
  
  \end{proof}
  
\section{Untangling embedded graphs}
  
  \label{untangling graphs}

  \begin{lemma}
  \label{l: bigon}Let $\Q$ be a finite set of points, $\Gamma=(\gamma_{1},\dots \gamma_{k})$ a $\mathcal{Q}$-embedded graph, $Y$ a compact subsurface of $X$ such that $\Gamma$ is transverse to $\partial Y$ and let $g\in H$ supported on $Y$ be such that $g_{\restriction Y}\in\H_{0}(Y,\partial Y)$ and
  $\Gamma$ and $\Gamma':=g\cdot \Gamma$ are transverse on $int(Y)$.
  Then for any $\V\in\nd$ there exists $h_{0}\in H^{+}_{[\Gamma]}$ and $h_{1}\in\V$ supported on $Y$ such that $h_{0}h_{1}\cdot\Gamma\simeq\Gamma'$.
  \end{lemma}
  \begin{proof}
  We prove the result by induction on $N:=|(\bigcup\Gamma)\cap(\bigcup\Gamma')\cap int(Y)|<\infty$.
  Pick some $\V_{0}=\V_{0}^{-1}\in\nd$ such that $\V_{0}^{3}\subseteq\V$.
  
  Let $\mathcal{A}$ be the collection of maximal subarcs of some $\gamma_{i}$ contained in $Y$. Notice that any two distinct
  $\alpha,\alpha'\in\mathcal{A}$ satisfy $\alpha\cap\alpha'=\emptyset$.
  If $N=0$ then for all $\alpha\in\mathcal{A}$ we have that $\alpha\cup g\circ\alpha$ is the boundary of some disk $D\subseteq Y$ with $\mathring{D}\cap(\bigcup\Gamma\cup\bigcup\Gamma')=\emptyset$.
  Then for any $\epsilon>0$ one can easily find some $h_{1}\in\V_{\epsilon}\cap H$ supported on $Y$ such that $h_{1}\cdot\alpha$ lies inside the closed bigon bounded by $\alpha$ and $\alpha'$ for all $\alpha\in\mathcal{A}$ and then some $h_{0}\in H^{+}_{\cup\Gamma}$ such that $h_{0}h_{1}\cdot\Gamma\simeq\Gamma'$.
  
  Assume now that $N>0$. Then by \ref{f: bigon removals} there are $1\leq i,j\leq k$ and a bigon $B$ in $Y$ whose interior is disjoint from $\Gamma\cup\Gamma'$ and which is bounded by the union of a subarc of $\gamma_{i}$ and a subarc of $\gamma'_{j}$ meeting only at the endpoints.
  
  It is easy to find a disk $D$ with $B\subseteq\mathring{D}$ intersecting
  $\Gamma$  in a single subarc $\alpha_{i}\subset\gamma_{i}$ and $\Gamma'$ in a single subarc  $\alpha_{j}'\subset\gamma'_{j}$ and containing no other points from $\bigcup\Gamma\cap\bigcup\Gamma'$ other than the two intersection points in $\partial B$. Then Fact \ref{f: arcs within disk} yields some $f$ supported on $D$ such that $f^{-1}\cdot\alpha'_{j}$ and $\alpha_{i}$ do not cross.
  
  We can then apply the induction hypothesis to $f^{-1}g$ and $\V_{0}$. We obtain $h_{0}'\in H^{+}_{[\Gamma]}$ and $h'_{1}\in\V_{0}$ supported on $Y$ such that $f^{-1}g\cdot\Gamma=h_{0}'h_{1}'\cdot\Gamma$.
  
  We claim that there exists $h''_{0}\in H^{+}_{[\Gamma]}$, $\tilde{h}\in H_{[h_{1}\cdot\Gamma]}$, $h''_{1}\in\V_{0}^{2}$ supported on $Y$ such that $f^{h_{0}'}=h_{0}''h''_{1}\tilde{h}$. A simple calculation then shows that $h_{0}:=h'_{0}h''_{0}\in H^{+}_{[\Gamma]}$ and $h_{1}:=h''_{1}h'_{1}\in\V_{0}^{3}\subseteq\V$ verify the properties we need.
  
  Notice that $f':=f^{h_{0}'}$ is supported on $D'=(h_{0}')^{-1}\cdot D$. It brings the arc
  $\beta''_{j}:=(fh_{0}')^{-1}\cdot\alpha'_{j}\subseteq h'_{1}\cdot\gamma_{j}\cap D$, which is disjoint from the arc $\beta_{i}:=(h_{0}')^{-1}\cdot\alpha_{i}\subseteq\gamma_{i}\cap D$, to a position in which it intersects the latter in exactly two points, creating a bigon inside $\mathring{D}'$. The proof now reduces to the following:
  \begin{claim}
  There is $h''_{1}\in\V_{0}^{2}$ supported in $D'$ bringing, $\beta''_{j}$ to a curve intersecting $\beta_{i}$ in exactly two points (bounding a bigon).
  \end{claim}
  Indeed, once such $h''_{1}$ is given, we first see using Fact \ref{f: arcs within disk} the existence of $\tilde{h}\in H^{+}_{h'_{1}\cdot\Gamma}$ supported on $D'$ such that
  $$(h''_{1}\tilde{h})^{-1}(\beta_{i})\cap\beta''_{j}=\tilde{h}^{-1}((h''_{1})^{-1}(\beta_{i})\cap\beta''_{j})=f'^{-1}(\beta_{i})\cap\beta''_{j}$$ and then the existence of
  $h''_{0}\in H^{+}_{[\Gamma]}$ such that $h''_{0}h''_{1}\tilde{h}=f'$ easily follows from the same fact.
  
  To prove the claim pick some disk $E$ in $D'$ which is divided into two smaller disks by an arc of $\beta_{i}$
  and has a diameter small enough so that any homeomorphism supported on $E$  belongs to $\V_{0}$.
  Choose points $p\in\beta''_{j}$, $q\in \mathring{E}$ in the same connected component $C_{0}$ of $\mathring{D'}\setminus\beta_{i}$, as well as some arc $\omega$ in $C_{0}$ from
  $p$ to $q$.
  
  Let $\epsilon>0$ be small enough that $N_{\epsilon}(\omega)\subseteq C_{0}$.
  By Lemma \ref{l: many point pushing maps} there exists some $\phi\in\p_{\epsilon}(\omega)\cap\V_{0}$.
  The arc $\phi\cdot\beta''_{j}$ does not intersect $\beta_{i}$, but it intersects $\mathring{E}$, so it is easy to see that we can choose some $\theta$ supported on $E$ such that
  $h''_{1}:=\theta\phi\in\V_{0}^{2}$ brings $\beta_{j}'$ into the desired configuration (we may assume $\partial E$ and $\phi\cdot\beta''_{j}$ are transverse).
  \begin{remark}
  We may not have simply chosen $\omega$ to be a point-pushing map along a path crossing $\beta_{i}$ and dispense of $\theta$ altogether, since the condition $\phi\in N_{\epsilon}(\omega)$ is too weak to guarantee multiplying on the left by $\phi$ creates only one new bigon.
  \end{remark}
  \end{proof}

  \begin{lemma}
  \label{l: flexible stabilizer} For any $\epsilon>0$ and any embedded graph $\Gamma=(\gamma_{j})_{j=1}^{k}$ we have
  $$H^{+}_{[\Gamma]}\subseteq\V_{\epsilon}H^{+}_{\cup\Gamma}\V_{\epsilon}H^{+}_{\cup\Gamma}.$$
  \end{lemma}
  \begin{proof}
  Let $\mathcal{Q}=V(\Gamma)$. It suffices to show that for any homeomorphism $h$ of $\cup\Gamma$ preserving the arcs of $\Gamma$ with orientation and fixing their endpoints there is some $g\in\V_{\epsilon}H^{+}_{\cup\Gamma}\V_{\epsilon}$ such that $g_{\restriction \Gamma}=h$.
  
  Fact \ref{f: nice neighbourhoods} there exists some family $\{D_{q}\}_{q\in\Q}$ of disks with $q\in\mathring{D}_{q}$ and $\delta\in(0,\frac{\epsilon}{10})$ such that $D_{q}$ intersects $\Gamma$ in a star,
  $diam(D_{q})<\frac{\epsilon}{4}$ and $B(q,5\delta)\subseteq D_{q}$.
  
  By continuity of $h$ there are subarcs
  $\hat{\gamma}_{j}\subseteq\gamma_{j}\setminus\mathcal{Q}$ such that:
  $$
  \gamma_{j}\setminus\hat{\gamma}_{j}\cup h\cdot(\gamma_{j}\setminus\mathring{\gamma_{j}})\subseteq N_{\delta}(\gamma_{j}(\{-1,1\}))
  $$
  By Facts \ref{f: extensions around branching points} and \ref{f: nice neighbourhoods} there exist some
  $\phi_{j}\in H$ supported in some small neighbourhood of $\bigcup_{j}\gamma_{j}\setminus\hat{\gamma}_{j}$
  such that $\phi_{j}\circ h$ preserves $\gamma_{j}$ and fixes
  the endpoints of $\hat{\gamma}_{j}$ (thus preserving $\hat{\gamma}_{j}$ as well). It is not difficult to see that one can choose $\phi_{j}\in\V_{4\delta}$ with disjoint supports, so that $\phi:=\prod_{j=1}^{k}\phi_{j}\in\V_{4\delta}$.
  
  The following easy consequence of Fact \ref{f: nice neighbourhoods} is left to the reader.
  \begin{claim}
  There exists some $g_{1}\in\V_{\delta}$ such that $g_{1}\cdot\hat{\gamma}_{j}\cap\bigcup\Gamma=\emptyset$ for all $1\leq j\leq k$.
  \end{claim}
  
  Consider now the partial homeomorphism $h'=(g_{1}\phi \circ h\circ g_{1}^{-1})_{\restriction \bigcup_{j}\hat{\gamma}_{j}}$.
  Notice that the fact that $g_{1}\cdot\gamma_{j}$ admits a rectangular extension implies the existence of some disk $D_{j}$
  disjoint from $\bigcup\Gamma$, with $\partial D_{j}$ transverse to $g_{1}\cdot\gamma_{j}$ and
  $D_{j}\cap g_{1}\cdot\gamma_{j}=g_{1}\cdot\hat{\gamma}_{j}$. Together with Fact \ref{f: arcs within disk}
  this implies that there is some $g_{2}\in H$ which extends $h'$, is orientation preserving at each arc $g_{1}\cdot\hat{\gamma}_{j}$ and is supported in the complement of $\bigcup\Gamma\cup\bigcup_{j=1}^{k}g_{1}\cdot(\gamma_{j}\setminus\hat{\gamma}_{j})$.
  In particular $g_{2}^{g_{1}}$ is the identity on $\bigcup_{j=1}^{k}(\gamma_{j}\setminus\hat{\gamma}_{j})$.
  
  On the one hand, $\phi\in\V_{4\delta}$ so for $\eta\in\{1,-1\}$ the component of $\gamma_{j}\setminus\hat{\gamma}_{j}$ in $B(\gamma_{j}(\eta),\delta)$ is mapped by $\phi$ into $B(\gamma_{j}(\eta),5\delta)\subseteq \mathring{D}_{q}$. On the other hand, $\hat{\gamma}_{j}$, so the image of said component is disjoint from $\bigcup_{j}\hat{\gamma}_{j}$. It follows from Fact \ref{f: extensions around branching points} that there exists some $g_{3}\in H$ supported on $\bigcup_{q\in\Q}D_{q}$ fixing $\bigcup_{j=1}^{k}\hat{\gamma}_{j}$ such that $(g_{3})_{\restriction S}=\phi_{\restriction S}$.
  Notice that $g_{3}\in\V_{\frac{\epsilon}{4}}$.
  The element  $\phi^{-1}g_{3}g_{1}^{-1}g_{2}g_{1}\in\V_{4\delta}\V_{\frac{\epsilon}{2}}\V_{\delta}H^{+}_{\cup\Gamma}\V_{\delta}\subseteq\V_{\epsilon}H^{+}_{[\Gamma]}\V_{\epsilon}$ agrees with $h$ on the entire $\bigcup\Gamma$ and is orientation preserving at $\Gamma$ so we are done.
  \end{proof}

\section{Concluding the proof}
  \label{conclusion}
  
  \begin{proof}
  [\textit{Proof} of Theorem \ref{t: main}]
  Pick some $\V_{0}=\V_{0}^{-1}\in\nd$ such that our fixed element $g_{0}$ does not belong to $\V_{0}^{8}$ and let $D_{0}$ be a disk on which $g_{0}$ is supported. By Corollary \ref{c: control on branch points} there is some embedded $\mathcal{Q}$-graph $\Gamma$ such that $H^{+}_{\cup\Gamma}\subseteq\V^{0}$, $\Q\cap D_{0}=\emptyset$ and $\Gamma$ is transverse to $\partial D_{0}$. Using Fact \ref{f: transversality} we can find some $f\in H$ supported on $D_{0}$ such that $fg_{0}\cdot\Gamma$ and $\Gamma$ are transverse on $\mathring{D}_{0}$.
  By Lemma \ref{l: bigon} there are $h_{0}\in H^{+}_{[\Gamma]}$, $h_{1}\in\V_{0}$ supported on $D_{0}$ such that $h_{0}h_{1}\cdot\Gamma\simeq\Gamma'$.
  
  The element $\psi:=h_{1}^{-1}h_{0}^{-1}fg_{0}$ is the identity in a neighbourhood of $\Q$ and satisfies $\psi\cdot\Gamma\simeq\Gamma$. It follows that $\psi\in H^{+}_{[\Gamma]}$.
  Finally, note that $H^{+}_{[\Gamma]}\subseteq\V_{0}^{3}$ by Lemma \ref{l: flexible stabilizer} so that $$g_{0}\in f^{-1}h_{0}h_{1}H^{+}_{[\Gamma]}\subseteq \V_{0}H^{+}_{[\Gamma]}\V_{0}H^{+}_{[\Gamma]}\subseteq\V_{0}^{8},$$ a contradiction.
  \end{proof}

  \bibliographystyle{plain}
  \bibliography{bibliography}

\begin{thebibliography}{10}

\bibitem{arens1946topologies}
Richard Arens.
\newblock Topologies for homeomorphism groups.
\newblock {\em American Journal of Mathematics}, 68(4):593--610, 1946.

\bibitem{chang2017minimum}
Xiao Chang and Paul Gartside.
\newblock Minimum topological group topologies.
\newblock {\em Journal of Pure and Applied Algebra}, 221(8):2010--2024, 2017.

\bibitem{dikranjan2014minimality}
Dikran Dikranjan and Michael Megrelishvili.
\newblock Minimality conditions in topological groups.
\newblock In {\em Recent progress in general topology III}, pages 229--327.
  Springer, 2014.

\bibitem{farb2011primer}
Benson Farb and Dan Margalit.
\newblock {\em A primer on mapping class groups (pms-49)}.
\newblock Princeton university press, 2011.

\bibitem{hirsch2012differential}
Morris~W Hirsch.
\newblock {\em Differential topology}, volume~33.
\newblock Springer Science \& Business Media, 2012.

\bibitem{kallman1986uniqueness}
Robert~R Kallman.
\newblock Uniqueness results for homeomorphism groups.
\newblock {\em Transactions of the American Mathematical Society},
  295(1):389--396, 1986.

\bibitem{mann2016automatic}
Kathryn Mann.
\newblock Automatic continuity for homeomorphism groups and applications.
\newblock {\em Geometry \& Topology}, 20(5):3033--3056, 2016.

\bibitem{mann2020automatic}
Kathryn Mann.
\newblock Automatic continuity for homeomorphism groups of noncompact
  manifolds.
\newblock {\em arXiv preprint arXiv:2003.01173}, 2020.

\bibitem{milnor2016morse}
John Milnor.
\newblock Morse theory.(am-51), volume 51.
\newblock In {\em Morse Theory.(AM-51), Volume 51}. Princeton university press,
  2016.

\bibitem{moise2013geometric}
Edwin~E Moise.
\newblock {\em Geometric topology in dimensions 2 and 3}, volume~47.
\newblock Springer Science \& Business Media, 2013.

\bibitem{newman1939elements}
Maxwell Herman~Alexander Newman.
\newblock {\em Elements of the topology of plane sets of points}.
\newblock Cambridge, 1939.

\bibitem{rosendal2008automatic}
Christian Rosendal.
\newblock Automatic continuity in homeomorphism groups of compact 2-manifolds.
\newblock {\em Israel Journal of Mathematics}, 166(1):349--367, 2008.

\end{thebibliography}
  
\end{document}